\documentclass[11pt]{smfart}
\usepackage{color}
\usepackage{amssymb,verbatim}
\usepackage{hyperref}
\usepackage{amsthm,array,amssymb,amscd,amsfonts,amssymb,latexsym, url}
\usepackage{amsmath}
\usepackage[all]{xy}
\usepackage[french]{babel}

\newcounter{spec}
{\end{list}}

\renewcommand{\P}{{\mathbf P}}

\newcommand{\Q}{{\mathbb Q}}

\newcommand{\R}{{\mathbb R}}

\renewcommand{\L}{{\mathbb L}}

\newcommand{\Br}{{\operatorname{Br \  }}}

\newcommand{\Spec}{{\operatorname{Spec  }}}

\renewcommand{\lim}{\varprojlim}

\numberwithin{equation}{section}

\newfont{\gothic}{eufb10}

\renewcommand{\qed}{{\hfill$\square$}}

\newtheorem{theo}{Th\'{e}or\`{e}me}[section]
\newtheorem{prop}[theo]{Proposition}

\newtheorem{lem}[theo]{Lemme}
\newtheorem{cor}[theo]{Corollaire}
\theoremstyle{definition}
\newtheorem{defi}[theo]{D\'efinition}
\theoremstyle{remark}
\newtheorem{rema}[theo]{Remarque}
\newtheorem{remas}[theo]{Remarques}

\newcommand{\bthe}{\begin{theo}}
\newcommand{\ble}{\begin{lem}}
\newcommand{\bpr}{\begin{prop}}
\newcommand{\bco}{\begin{cor}}
\newcommand{\bde}{\begin{defi}}
\newcommand{\ethe}{\end{theo}}
\newcommand{\ele}{\end{lem}}
\newcommand{\epr}{\end{prop}}
\newcommand{\eco}{\end{cor}}
\newcommand{\ede}{\end{defi}}

\def\A{{\mathbb A}}

\def\Br{{\rm Br}}

\DeclareFontFamily{U}{wncy}{}
\DeclareFontShape{U}{wncy}{m}{n}{%
<5>wncyr5%
<6>wncyr6%
<7>wncyr7%
<8>wncyr8%
<9>wncyr9%
<10>wncyr10%
<11>wncyr10%
<12>wncyr6%
<14>wncyr7%
<17>wncyr8%
<20>wncyr10%
<25>wncyr10}{}
\DeclareMathAlphabet{\cyr}{U}{wncy}{m}{n}
\begin{document}

  \title[Saut du rang] {Point g\'en\'erique et saut du rang du groupe de Mordell-Weil}

\author{J.-L. Colliot-Th\'el\`ene}
\address{Universit\'e Paris Sud, Universit\'e Paris-Saclay\\Math\'ematiques, B\^atiment 307\\91405 Orsay Cedex\\France}
\email{jlct@math.u-psud.fr}

 \date{soumis  le 14 ao\^{u}t 2019; version l\'eg\`erement  r\'evis\'ee soumise le 26 f\'evrier 2020}

\maketitle

{\it \`A la m\'emoire d'Andr\'e N\'eron}

\bigskip

\begin{abstract}
Soient $k$ un corps de nombres et $U$ une $k$-vari\'et\'e lisse int\`egre.
Soit $X \to U$ une famille de vari\'et\'es ab\'eliennes. On \'etablit les \'enonc\'es suivants.
Si la $k$-vari\'et\'e $X$ est domin\'ee par une $k$-vari\'et\'e
qui satisfait l'approximation faible faible, par exemple si $X$ est $k$-unirationnelle,
alors l'ensemble $\mathcal{R}$  des $k$-points
de $U$ dont la fibre a un rang de Mordell-Weil strictement plus grand
que celui de la fibre g\'en\'erique 
est dense dans $U$ pour la topologie de Zariski.
Si $X$ est $k$-rationnelle, alors $\mathcal{R}$ n'est pas mince dans $U$.
Ceci g\'en\'eralise des r\'esultats de Billard et de Salgado.
L'id\'ee principale de la d\'emonstration, id\'ee qui remonte \`a la th\`ese de N\'eron  \cite{N}, est d'utiliser le point g\'en\'erique de la fibre g\'en\'erique
de la famille,
et d'appliquer ensuite le th\'eor\`eme de sp\'ecialisation de N\'eron.
Dans l'appendice, on voit comment la m\'ethode du point g\'en\'erique
donne des  r\'esultats de  Billing, de N\'eron, et de Holmes et Pannekoek.
 \end{abstract}
 
 \begin{altabstract}
 Let $k$ be a number field and $U$ a smooth integral $k$-variety.
 Let $X \to U$ be an abelian scheme. 
We consider the set $\mathcal{R}$ of rational points $m \in U(k)$ such that the
 Mordell-Weil rank of the fibre $U_{m}$ is strictly bigger than the Mordell-Weil rank
 of the generic fibre. We prove the following results.
 If the $k$-variety $X$ is $k$-unirational, then $\mathcal{R}$ is 
 dense for the Zariski
 topology on $U$. If $X$ is $k$-rational, then  $\mathcal{R}$ is not thin in $U$.
 \end{altabstract}

 \section{Introduction}
 
Soient $k$ un corps de nombres, $U$ une $k$-vari\'et\'e lisse g\'eom\'etriquement int\`egre,
et $\pi : X \to U$ une famille de vari\'et\'es ab\'eliennes, c'est-\`a-dire un $U$-sch\'ema ab\'elien.
Soit $F=k(U)$ le corps des fonctions rationnelles de $U$, et soit $X_{\eta}/F$ la $F$-vari\'et\'e  ab\'elienne donn\'ee par la fibre g\'en\'erique de $\pi$.
Soit $r$ le rang g\'en\'erique, c'est-\`a-dire le rang du groupe $X_{\eta}(F)$, groupe de type fini d'apr\`es Mordell, Weil et N\'eron.

On renvoie au pagrapraphe 1 pour les rappels sur la d\'efinition d'un sous-ensemble mince
de l'ensemble $W(k)$ des points $k$-rationnels d'une $k$-vari\'et\'e $W$.

Un des th\'eor\`emes principaux de la th\`ese de N\'eron \cite{N} assure, sans hypoth\`ese sur $U$,  que l'ensemble des points $m \in U(k)$ dont la fibre $X_{m}$ a son groupe de Mordell-Weil $X_{m}(k)$ de rang strictement plus petit que $r$ est un ensemble mince dans $U$
(voir le th\'eor\`eme \ref{neronserre} ci-dessous).
L'\'enonc\'e plus connu, et qui en est une cons\'equence, est que si la $k$-vari\'et\'e  $U$ est $k$-rationnelle, i.e. $k$-birationnelle \`a un espace projectif,
alors l'ensemble des points $m \in U(k)$ dont la fibre $X_{m}$ a rang de Mordell-Weil au moins \'egal \`a $r$ n'est pas mince dans $U$, et en particulier est   dense dans  $U$ pour la topologie de Zariski.

Dans le pr\'esent article, on s'int\'eresse \`a l'ensemble des points $m \in U(k)$ dont la fibre $X_{m}$ a rang de Mordell-Weil au moins \'egal \`a $r+1$.

Commen\c cons par rappeler que, pour une $k$-vari\'et\'e lisse 
g\'eom\'etriquement int\`egre $W$,
chacun des \'enonc\'es ci-dessous implique le suivant
(les d\'efinitions des diverses propri\'et\'es sont rappel\'ees au paragraphe \ref{rappels}).

$\bullet$  La $k$-vari\'et\'e $W$ est $k$-rationnelle.
 
$\bullet$  La $k$-vari\'et\'e $W$ satisfait l'approximation faible faible.

$\bullet$  L'ensemble $W(k)$ n'est pas mince dans $W$.

$\bullet$ L'ensemble $W(k)$ est   dense dans $W$ pour la topologie de Zariski.

\medskip

Le th\'eor\`eme principal de cet article est le th\'eor\`eme \ref{Hauptsatz}, dont l'\'enonc\'e 
couvre un grand nombre d'implications. Voici deux cas particuliers.

\begin{theo}
Soient $k$ un corps de nombres, $U$ une $k$-vari\'et\'e lisse g\'eom\'etriquement int\`egre,
et $\pi : X \to U$ une famille de vari\'et\'es ab\'eliennes de dimension relative au moins 1.
Soit $r$ le rang de Mordell-Weil de la fibre g\'en\'erique. 
Soit ${\mathcal R} \subset U(k)$ l'ensemble des points de $U(k)$ dont la fibre $X_{m}/k$ a
rang de Mordell-Weil au moins \'egal \`a $r+1$.

Chacun des \'enonc\'es ci-dessous implique le suivant.

(a) $X$ est $k$-rationnelle,
 
 (b) $X$ satisfait l'approximation faible faible,

(c) $X(k)$ n'est pas mince dans $X$,

(d) ${\mathcal R}$ n'est pas mince dans $U$.

(e) ${\mathcal R}$ est  dense dans $U$ pour la topologie de Zariski.
\end{theo}

Ce th\'eor\`eme s'applique par exemple
\`a la surface  cubique $X$ donn\'ee par  
l'\'equation homog\`ene $x^3+y^3+z^3+t^3=0$,
qui est $k$-rationnelle (Euler), 
qu'on fibre (rationnellement)  
en coupant par $z=\lambda t$, ce qui donne la famille
de courbes elliptiques
$$x^3+y^3+(\lambda^3+1)t^3=0$$
sur (un ouvert de) $\Spec(k[\lambda])$,
avec section nulle donn\'ee par $(x,y,t)=(1,-1,0)$.

\begin{theo}
Soient $k$ un corps de nombres, $U$ une $k$-vari\'et\'e lisse g\'eom\'etriquement int\`egre,
et $\pi : X \to U$ une famille de vari\'et\'es ab\'eliennes de dimension relative au moins 1. 
Soit $r$ le rang de Mordell-Weil de la fibre g\'en\'erique. 
Soit ${\mathcal R} \subset U(k)$ l'ensemble des points de $U(k)$ dont la fibre $X_{m}/k$ a
rang de Mordell-Weil au moins \'egal \`a $r+1$.

 1) Si la $k$-vari\'et\'e $X$ est $k$-unirationnelle, l'ensemble ${\mathcal R}$ est   dense dans $U$ pour la topologie de Zariski.

2) S'il existe une $k$-vari\'et\'e lisse   $W$  dont l'ensemble $W(k)$ des $k$-points n'est pas mince dans $W$, ce qui est le cas si $W$ est stablement $k$-rationnelle,  ou si elle satisfait l'approximation faible fabile, et s'il existe un $k$-morphisme
$W \to X$ dominant tel que la fibre g\'en\'erique
de l'application compos\'ee $W \to X \to U$ est g\'eom\'etriquement int\`egre, alors
${\mathcal R}$ n'est pas mince dans $U$.
\end{theo}

Ce th\'eor\`eme s'applique aux familles de courbes elliptiques d'\'equation
affine
$$ d(\lambda) y^2= p(x)$$
au-dessus de (un ouvert de)   $\Spec(k[\lambda])$, avec $d(\lambda)$ polyn\^{o}me s\'eparable de degr\'e 2 et $p(x)$ polyn\^{o}me
s\'eparable de degr\'e 3. Une telle surface est $k$-birationnelle \`a une surface de Ch\^{a}telet
\cite{CTSaSD}.
On prend pour $W \to X$ un torseur universel sur un mod\`ele projectif et lisse de la surface, de fibre triviale en un point $k$-rationnel. D'apr\`es \cite[Thm. 8.1]{CTSaSD}, un tel torseur est une vari\'et\'e stablement $k$-rationnelle.

\medskip

Pour $\pi: X\to U$ une famille de vari\'et\'es
ab\'eliennes de dimension relative au moins 1, on dispose de l'inclusion de corps de fonctions
$k(U) \subset k(X)$.
La d\'emonstration du th\'eor\`eme   \ref{Hauptsatz} et en particulier des deux  th\'eor\`emes ci-dessus repose sur un lemme  tr\`es simple,  le lemme \ref{generique},
dont l'esprit remonte \`a la th\`ese de N\'eron :  le rang du groupe de type fini $X_{\eta}(k(X))$
est strictement sup\'erieur au rang du groupe de Mordell-Weil-N\'eron $X_{\eta}(k(U))$ de la fibre g\'en\'erique de $\pi$, car le point g\'en\'erique de $X$
d\'efinit un point de $X_{\eta}(k(X))$ dont aucun multiple n'appartient \`a $X_{\eta}(k(U))$.

\medskip

La m\'ethode propos\'ee ici  permet de retrouver, et de g\'en\'eraliser,
divers r\'esultats sur le saut du rang :

$\bullet$  Certains r\'esultats de Billard \cite{B}, Salgado \cite{S}, Hindry-Salgado \cite{HS}, Loughran-Salgado \cite{LS} (voir les remarques \ref{bilsalhinlou}).

$\bullet$ Un r\'esultat de Rohrlich \cite{R} (voir la Proposition \ref{degre2}), 
 
$\bullet$ Un ancien r\'esultat de Billing et N\'eron,
et un r\'esultat de Holmes--Pannekoek \cite{HS}
(voir  le \S \ref{grandrang}).

 Il convient cependant de remarquer que les m\'ethodes de
 Salgado, Hindry-Salgado, Loughran-Salgado sur les surfaces elliptiques
 g\'eom\'etriquement rationnelles, qui utilisent des syst\`emes lin\'eaires de coniques
 sur ces surfaces, permettent sous certaines hypoth\`eses d'\'etablir pour ces surfaces
 des sauts du rang   d'au moins 2 par rapport \`a la fibre g\'en\'erique,
 ce que ne donne pas la pr\'esente m\'ethode.
 
 Je d\'edie ce texte \`a la m\'emoire d'Andr\'e N\'eron, qui fut mon directeur de th\`ese.

\section{Rappels}\label{rappels}

Soit $k$ un corps. Par d\'efinition, une $k$-vari\'et\'e est un $k$-sch\'ema s\'epar\'e de type fini. On note $X(k)$ l'ensemble des points $k$-rationnels de $X$.
Une $k$-vari\'et\'e lisse int\`egre qui poss\`ede un point $k$-rationnel est
g\'eom\'etriquement int\`egre.

Soit $k$ un corps de caract\'eristique z\'ero. Soit   $X$ une $k$-vari\'et\'e int\`egre
de dimension au moins 1.
Suivant Serre \cite{SMW,SGT}, un sous-ensemble $E \subset X(k)$  est dit mince s'il existe une $k$-vari\'et\'e  r\'eduite $Y$ et un
$k$-morphisme $f : Y \to X$  tels que :

(a) $E \subset f(Y(k))$.

(b) La fibre g\'en\'erique de $f$ est finie et n'admet pas de section rationnelle.

\medskip

Une union finie de sous-ensembles minces de $X(k)$ est mince. En particulier,
si $X(k)$ n'est pas mince dans $X$, alors le compl\'ementaire d'un sous-ensemble mince
de $X(k)$ n'est pas mince dans $X$. Si $E \subset X(k)$ n'est pas mince dans $X$, alors
$E$ est dense dans $X$ pour la topologie de Zariski. La r\'eciproque est clairement 
fausse.

\begin{lem}\label{mincegeomint} Soient $k$ un corps de caract\'eristique z\'ero et   $\pi : Z \to~X$ un $k$-morphisme
dominant de $k$-vari\'et\'es int\`egres. Supposons la fibre g\'en\'erique de $\pi$ g\'eom\'etriquement 
int\`egre. Soit $E\subset Z(k)$.  Si $\pi(E) \subset X(k)$ est mince dans $X$, alors $E$ est mince dans $Z$.  
\end{lem} 
\begin{proof}
Soit $f : Y \to X$ comme ci-dessus, avec $\pi(E) \subset f(Y(k))$. 
Pour \'etablir l'\'enonc\'e, on peut remplacer
$Z$ par un un ouvert non vide. On peut donc supposer que la $k$-vari\'et\'e  $Y \times_{X}Z $  est r\'eduite.
Sous les hypoth\`eses du lemme, la fibre g\'en\'erique de $g : W:=Y \times_{X}Z  \to Z$
est finie et n'admet pas de section rationnelle. On a $E \subset g(W(k))$. $\Box$.
\end{proof}

\medskip

Soit $k$ un corps de nombres. Soit $\Omega$ l'ensemble de ses places.
Soit $X$ une $k$-vari\'et\'e lisse   int\`egre poss\'edant un point $k$-rationnel.
On consid\`ere l'application diagonale
$$  \rho: X(k) \to \prod_{v\in \Omega} X(k_{v})$$
de $X(k)$ dans le produit topologique des ensembles $X(k_{v})$ munis de la topologie
induite par celle de $k_{v}$,
et, pour tout ensemble $S$  de places, l'application diagonale
$$\rho_{S} : X(k) \to \prod_{v\in S} X(k_{v}).$$
Si $X$ est projective, l'espace topologique $ \prod_{v\in \Omega} X(k_{v})$
co\"{\i}ncide avec l'espace des ad\`eles $X(\A_{k})$.

On va s'int\'eresser aux propri\'et\'es suivantes.

(AF) (Approximation faible) L'ensemble $X(k)$ est dense dans
 le produit topologique $\prod_{v \in \Omega} X(k_{v})$. Ceci \'equivaut \`a dire que, pour tout ensemble fini $S$ de places,
l'application $\rho_{S}$ a son image dense.  

(AFO) (Approximation faible ouverte)  L'adh\'erence de l'image de $\rho$ est ouverte. Ceci \'equivaut \`a dire
que, pour tout ensemble fini $S$ de places, l'adh\'erence de $\rho_{S}$
est ouverte, et qu'il existe un ensemble fini $S_{0}$ de places  tel que,
pour $S \subset \Omega$ fini avec $S\cap S_{0} = \emptyset$,
l'image de $\rho_{S}$ est dense.

(AFF) (Approximation faible faible)  Il existe un ensemble fini $S_{1}$ de places  tel que, pour
$S \subset \Omega$ fini avec $S\cap S_{1} = \emptyset$,
l'image de $\rho_{S}$ est dense.  

(BMF) (Brauer--Manin) La $k$-vari\'et\'e $X$ est projective et le quotient de groupes de Brauer $\Br(X)/\Br(k)$
est fini, ce qui assure que le ferm\'e de Brauer-Manin
$X(\A_{k})^{\Br(X)} \subset  X(\A_{k})$ est ouvert,  
et l'image de $\rho$ co\"{\i}ncide avec cet ouvert-ferm\'e.

Pour $X$ projective, lisse, g\'eom\'etriquement int\`egre sur un corps $k$ de  
 caract\'eristique z\'ero,
l'hypoth\`ese que le quotient $\Br(X)/\Br(k)$  est fini est satisfaite si
$X$ est g\'eom\'etriquement unirationnelle, et plus g\'en\'eralement si $X$
est g\'eom\'etriquement rationnellement connexe.  Sur $k$  un corps de nombres, 
elle est  satisfaite pour les surfaces $K3$ \cite{SZ} et les vari\'et\'es de Kummer \cite{HP}.
 
 On a les implications : (AF) implique (AFO) qui implique (AFF); 
 sous l'hypoth\`ese de (BMF), on a (AFO) et donc (AFF).
 
 \medskip
 
 Si $X$ est une $k$-vari\'et\'e lisse int\`egre  et $U \subset X$
un ouvert non vide poss\'edant un $k$-point,
chacune des  propri\'et\'es (AF), (AFO), (AFF)  vaut pour $X$ si et seulement si elle vaut pour $U$. Si $X$ est de plus projective et  $\Br(X)/\Br(k)$ est fini, 
la propri\'et\'e (BMF) pour $X$ implique les propri\'et\'es (AFO) et (AFF) pour $U$.

\medskip

 Serre  \cite[\S 3, Thm. 3.5.3]{SGT} a \'etabli :
\begin{theo}\label{serremincefin}
Soit $k$ un corps de nombres. Soit $X$ une $k$-vari\'et\'e  lisse et
g\'eom\'etriquement int\`egre. 
Soit $E \subset X(k)$ un ensemble mince.
Soit $S_{0}$ un ensemble fini de places. Il existe un ensemble fini    $S$ de places
 avec $S \cap S_{0} = \emptyset$ tel que l'image de $E$
dans $\prod_{v \in S} X(k_{v})$ n'est pas dense.
\end{theo}
On en d\'eduit l'\'enonc\'e suivant (cf. \cite[Cor. 3.5.4]{SGT}).
 \begin{theo}\label{mincefin}
Soit $k$ un corps de nombres. Soit $X$ une $k$-vari\'et\'e  lisse et
g\'eom\'etriquement 
int\`egre. Supposons 
$X(k)\neq \emptyset$.
Soit $E \subset X(k)$ un ensemble mince.

(i) Si $X$ satisfait (AF), alors l'image de $X(k) \setminus E$ dans
$\prod_{v \in \Omega} X(k_{v})$ est dense.

(ii) Si $X$ satisfait (AFO), alors 
 l'adh\'erence de l'image de $X(k) \setminus E$ dans
$\prod_{v \in \Omega} X(k_{v})$ co\"{\i}ncide avec celle de X(k), et est ouverte.

(iii) Si $X$ satisfait (AFF), alors il existe un ensemble fini $S_{1}$ de places
tel que
 $X(k) \setminus E$ est dense  
dans $\prod_{v \notin S_{1}} X(k_{v})$.

(iv) Si $X$ est projectif et satisfait (BMF), alors l'adh\'erence de $X(k) \setminus E$
dans $X(\A_{k})$ co\"{\i}ncide  avec celle de $X(k)$, \`a savoir $X(\A_{k})^{\Br(X)}$, et elle est ouverte
dans $X(\A_{k})$.

(v) Dans tous ces cas,  $X(k) \setminus E$ n'est pas mince, et  est dense dans $X$ pour la topologie de Zariski.
\end{theo}
\begin{proof}
(i) Soit $S_{0}$ un ensemble fini de places et $U_{0} \subset \prod_{v\in S_{0}} X(k_{v})$ un ouvert non vide. 
D'apr\`es le th\'eor\`eme    \ref{serremincefin}, il existe  
  un ensemble fini $S$ de places  avec $S \cap S_{0} = \emptyset$
et $\prod_{v\in S} U_{v} \subset \prod_{v\in S} X(k_{v})$ un ouvert qui ne rencontre pas
l'image diagonale de $E$. Alors 
$U_{0} \times \prod_{v\in S} U_{v}$ ne rencontre pas
l'image diagonale de $E$. 

Si $X$ satisfait (AF), pour tout ensemble $S_{0}$ et tout ouvert non vide
$U_{0} \subset \prod_{v\in S_{0}} X(k_{v})$
l'image diagonale de $X(k)$ 
est dense dans 
$$U_{0} \times \prod_{v\in S} U_{v} \subset \prod_{v\in  S \cup S_{0} }   X(k_{v}),$$ 
 et donc il en est de m\^eme
de $X(k) \setminus E$. Ceci implique que l'image diagonale de $X(k)\setminus E$
est dense dans $U_{0} \subset \prod_{v\in S_{0}} X(k_{v})$.

Si $X$ satisfait (AFO), il existe un ensemble fini $S_{0}$ de places et un ouvert non vide
$U_{0} \subset \prod_{v\in S_{0}} X(k_{v})$ tel que l'adh\'erence de 
 l'image diagonale de $X(k)$ dans $\prod_{v \in \Omega}X(k_{v})$ soit
  $U_{0} \times \prod_{v \notin S_{0}} X(k_{v})$. On conclut avec le m\^{e}me argument que ci-dessus.

Si $X$ satisfait (AFF),  il existe un ensemble fini $S_{1}$ de places tel que
l'image diagonale de $X(k)$ dans $\prod_{v \notin S_{1}} X(k_{v})$ est dense.
On raisonne comme en (i) en rempla\c cant
 $\prod_{v \in \Omega} X(k_{v})$ par $\prod_{v \notin S_{1}} X(k_{v})$.
 
  Sous l'hypoth\`ese (BM) pour $X$, on a  (AFO).
\end{proof}

Le cas des $k$-vari\'et\'es $k$-rationnelles, i.e. $k$-birationnelles \`a un espace projectif sur $k$,  est le cas classique.

Sur un corps de nombres $k$, la propri\'et\'e (BMF) pour une $k$-vari\'et\'e projective, lisse, g\'eom\'etriquement int\`egre
$X$ avec $X(k) \neq \emptyset$,
 et donc aussi les propri\'et\'es   (AFO), (AFF) pour les ouverts non vides de $X$,
sont conjectur\'ees pour toutes les vari\'et\'es $X$ g\'eom\'etriquement rationnellement connexes. L'\'enonc\'e (BMF), et donc aussi les \'enonc\'es (AFO) et (AFF),
ont \'et\'e \'etablis pour les compactifications lisses d'un certain nombre de
vari\'et\'es :

(1) Espaces homog\`enes de $k$-groupes lin\'eaires connexes lorsque les stabilisateurs
sont connexes (Sansuc, Borovoi).

(2) Surfaces de Ch\^{a}telet, donn\'ees par une \'equation affine $y^2-az^2=P(x)$
avec $a \in k^{\times}$ et $P(x)$ polyn\^{o}me s\'eparable de degr\'e 3 ou 4 \cite[Thm. 8.7]{CTSaSD}.

(3) Surfaces fibr\'ees en coniques sur ${\bf P}^1_{k}$
avec un $k$-point et avec au plus 4 fibres g\'eom\'etriques singuli\`eres \cite[Thm. 2]{CT}.

 (4)  Surfaces de del Pezzo de 4
avec un $k$-point (Salberger et Skorobogatov \cite[Thm. (0.1)]{SaSk}).

 (5) Surfaces fibr\'ees en coniques sur ${\bf P}^1_{k}$
 avec 5 fibres g\'eom\'etriques singuli\`eres. 
La d\'emonstration de \cite{SaSk} proc\`ede  
par r\'eduction \`a ce cas. On peut aussi d\'eduire le r\'esultat formellement de l'\'enonc\'e
\cite[Thm. (0.1)]{SaSk} en utilisant le th\'eor\`eme d'Iskovskikh \cite[Thm. 4]{Isk} que sur un corps
quelconque, une telle surface n'est pas $k$-minimale.

(6) Pour $k=\Q$ le corps des rationnels, les familles  de vari\'et\'es obtenues par Browning, Matthiesen et Skorobogatov \cite{BMS}  et Harpaz et Wittenberg \cite{HW},
 par exemple  les vari\'et\'es sur $\Q$ donn\'ees par une \'equation
 ${\rm Norm}_{K/{\Q}}(\Xi)=P(x)$
 avec $K/{\Q}$ une extension finie et
 $P(x) \in {\Q}[x]$ un polyn\^{o}me de degr\'e quelconque ayant toutes ses racines dans $\Q$.
 
On consultera le rapport r\'ecent \cite{W}.  L'\'enonc\'e (AFO), et donc aussi (AFF), a \'et\'e \'etabli pour les surfaces cubiques lisses par Swinnerton-Dyer
 \cite{SD}. Ceci implique que pour une telle surface $X$ avec $X(k)\neq \emptyset$, l'ensemble
 $X(k)$ n'est pas mince dans $X$. Ce dernier r\'esultat avait \'et\'e \'etabli par un argument de hauteurs
 par Manin \cite{Ma1} \cite[Chap. VI, Thm 46.1]{Ma2}.
  
  Soient $k$ un corps, $X$ une $k$-vari\'et\'e  lisse g\'eom\'etriquement int\`egre,
 et $Y \to X$ un sch\'ema ab\'elien.
Soit $K=k(X)$ le corps des fonctions de $X$. Soit $Y_{\eta}/K$ la fibre g\'en\'erique.
Pour tout $k$-point $m$ de $X$, de fibre $Y_{m}$, on dispose de l'application 
de sp\'ecialisation
$$ sp_{m} : Y_{\eta}(K) = Y(X) \to Y_{m}(k).$$

Dans \cite[\S 11]{SMW}, Serre \'etablit le th\'eor\`eme suivant, qui a pour
cons\'equence le th\'eor\`eme de sp\'ecialisation  
 de N\'eron \cite[Thm. 6]{N} :

\begin{theo}\label{neronserre}
Avec les notations ci-dessus, si $k$ est un corps de nombres,
 l'ensemble des points $m \in X(k)$ tels
 que l'application de sp\'ecialisation $sp_{m}$
 n'est pas injective est mince.
\end{theo}

Serre \'enonce le r\'esultat pour $X$ une vari\'et\'e $k$-rationnelle, 
mais cette hypoth\`ese ne joue aucun r\^{o}le dans la d\'emonstration,
seul importe le fait que le groupe des points rationnels de la fibre g\'en\'erique
sur le corps de type fini $K$ est un groupe de type fini (th\'eor\`eme de
Mordell-Weil-N\'eron \cite{N}).
La $k$-rationalit\'e de $X$ est appliqu\'ee ensuite pour voir que le compl\'ementaire
dans $X(k)$ d'un ensemble mince est dense.

 On peut  combiner les th\'eor\`emes \ref{mincefin} et \ref{neronserre}  pour obtenir des \'enonc\'es dont l'un des plus simples est le suivant : dans le th\'eor\`eme \ref{mincefin}, si $X(k) \neq \emptyset$ et $X$   satisfait l'approximation faible faible,
 par exemple si $X$ est $k$-birationnelle \`a un espace projectif,  alors  il existe un ensemble fini $S_{0}$ de places de $k$ tel que  pour tout ensemble fini $S$ de places avec $S \cap S_{0}= \emptyset$, les $k$-points de $X$ pour lesquels  $sp_{m}$ est injective sont denses dans $\prod_{v \in S} X(k_{v})$.

 \section{Saut du rang}
 
Nos r\'esultats reposent sur le lemme-cl\'e suivant.
 
\ble\label{generique} Soient $F$ un corps et $A/F$ une vari\'et\'e ab\'elienne.
Soit
  $x \in A$ un point sch\'ematique de $A$ qui n'est
  pas un point ferm\'e.
Soit $F(x)$ le corps r\'esiduel.
Alors le quotient $A(F(x))/A(F)$ n'est pas un groupe de torsion.
\ele
\begin{proof} Il suffit de traiter le cas
  o\`u $F$ est alg\'ebriquement clos, auquel cas le groupe
  $A(F)$ est divisible.
Soit  $z \in A(F(x))$ un point  qui n'est pas dans $A(F)$.
Supposons qu'il existe un entier $m \geq 1$ tel que
  $mz \in A(F)$. Alors il existe $z_{0} \in A(F)$
avec $mz=mz_{0}$, et donc $m(z-z_{0})=0$.
Comme  la $m$-torsion de $A(F)$ co\"{\i}ncide avec
celle de $A(F(x))$, ceci donne $z\in A(F)$, ce qui est une contradiction.
\end{proof}

 \begin{rema} Ce lemme s'applique en particulier au point g\'en\'erique
 de $A$ si ${\rm dim}(A)\geq 1$.
 Citons ici  l'\'enonc\'e de N\'eron   \cite[Prop. 7, p. 156]{N}) {\it verbatim}.
  {\it Soient $C$ une courbe de genre $g\geq 1$ d\'efinie sur un corps $k$
  et $A=(A_{1},\dots, A_{m})$ un syst\`eme de $m$ points g\'en\'eriques
  ind\'ependants sur $C$. Alors $C$ est de rang r\'eduit $\geq m$ dans 
  $k(A)=k(A_{1},\dots,A_{m})$.}
 \end{rema}

 Soient $k$ un corps de nombres, $U$ une $k$-vari\'et\'e  lisse g\'eom\'etriquement int\`egre 
 de dimension au moins $1$ poss\'edant un $k$-point,
 et $\pi : X \to U$ un sch\'ema ab\'elien de dimension relative au moins 1.
 Soit $W $ une $k$-vari\'et\'e lisse g\'eom\'etriquement int\`egre avec un $k$-point
 et $g: W \to X$ un $k$-morphisme.
 
 Soit $r$ le rang   de Mordell-Weil de la fibre g\'en\'erique $X_{\eta}$ sur le corps $K=k(U)$.
 Soit ${\mathcal R}  \subset U(k)$ l'ensemble des points  $m\in U(k)$ 
 tels que le rang $r_{m}$ de la fibre $X_{m}$ sur $k$
  soit au moins \'egal \`a $r+1$.
  
  \medskip
  
 Consid\'erons les hypoth\`eses   :

\medskip

 (1) La $k$-vari\'et\'e $W$ est $k$-rationnelle. 
 
 (2)  La $k$-vari\'et\'e $W$ satisfait l'approximation faible. 

(3)  La $k$-vari\'et\'e $W$ satisfait l'approximation faible ouverte.

(4)   La $k$-vari\'et\'e $W$ satisfait l'approximation faible faible.

(5)  L'ensemble $W(k)$ n'est pas mince dans la $k$-vari\'et\'e $W$.

 (6)  La $k$-vari\'et\'e $X$  satisfait l'approximation faible.   
 
   (7)   La $k$-vari\'et\'e $X$  satisfait l'approximation faible ouverte.  

 (8)  La $k$-vari\'et\'e $X$ satisfait l'approximation faible faible.  
 
 (9)   L'ensemble $X(k)$ n'est pas mince dans   la $k$-vari\'et\'e $X$.
 
\smallskip
 
 et les \'enonc\'es  :
 
 \medskip

(a) L'ensemble ${\mathcal R}$ est dense  dans $\prod_{v \in \Omega} U(k_{v})$,
c'est-\`a-dire que, pour tout ensemble fini $S$ de places, ${\mathcal R}$ est dense dans $\prod_{v\in S} U(k_{v})$.

(b)   L'adh\'erence de ${\mathcal R}$ dans $\prod_{v \in \Omega} U(k_{v})$ est ouverte,
c'est-\`a-dire qu'il existe un ensemble fini $S_{0}$ de places de $k$ tel que,
pour tout ensemble fini $S$ de places, l'adh\'erence de ${\mathcal R}$ dans
$\prod_{v\in S} U(k_{v})$ est ouverte, et qu'elle est \'egale \`a $\prod_{v\in S} U(k_{v})$
quand on a  $S \cap S_{0} = \emptyset$.

(c) Il existe un ensemble fini $S_{0}$ de places tel que
${\mathcal R}$  est dense dans $\prod_{v \notin S_{0}} U(k_{v})$.

(d) Pour  tout ensemble fini non vide  $S$ de places de $k$,
 l'adh\'erence de ${\mathcal R}$ dans $\prod_{v\in S} U(k_{v})$ contient un ouvert
non vide.

 (e)  Il existe un ensemble fini $S_{0}$ de places de $k$ tel que, pour
 tout ensemble fini  non vide $S$ de places de $k$ avec $S\cap S_{0} = \emptyset$,
 l'adh\'erence de ${\mathcal R}$ dans $\prod_{v\in S} U(k_{v})$ contient un ouvert
non vide.

(f) L'ensemble ${\mathcal R}$ n'est pas mince dans $U$.

(g) L'ensemble ${\mathcal R}\subset U(k)$ est dense dans $U$ pour la topologie de Zariski.

 \medskip
 
 D'apr\`es le th\'eor\`eme \ref{serremincefin}, on a  $(4) \Longrightarrow (5)$ et
 $(8) \Longrightarrow (9)$.  D'apr\`es 
 le th\'eor\`eme \ref{mincefin},
 chacune des hypoth\`eses (a), (b), (c) implique (f).
 Comme remarqu\'e au d\'ebut du \S \ref{rappels}, (f) implique (g).

 Les implications suivantes sont \'evidentes :  
  $(1) \Longrightarrow (2) \Longrightarrow (3) \Longrightarrow (4)$  et
 $(6)  \Longrightarrow (7)   \Longrightarrow (8) $.
 Il en est de m\^{e}me de  :
$(a) \Longrightarrow  (b) \Longrightarrow (d) \Longrightarrow (e)$
  et
 $(c)  \Longrightarrow (e).$

Si $X$ est $k$-unirationnelle, dans ce qui suit on peut
supposer (1) pour $W$.
  Si $X$ est $k$-rationnelle, on a (6).

   \medskip

 Voici le th\'eor\`eme principal de cette note.

 \begin{theo}\label{Hauptsatz}
 Soient $k$ un corps de nombres, $U$ une $k$-vari\'et\'e  lisse g\'eom\'etriquement int\`egre 
 de dimension au moins $1$ poss\'edant un $k$-point,
 et $\pi : X \to U$ un sch\'ema ab\'elien de dimension relative au moins 1.
 Soit $W $ une $k$-vari\'et\'e lisse g\'eom\'etriquement int\`egre avec un $k$-point
 et $g: W \to X$ un $k$-morphisme.
 Soit $g(W)_{ad} \subset X$ l'image sch\'ematique du morphisme $W \to X$.
 Supposons que le morphisme    $g(W)_{ad} \to U$ induit par $\pi : X \to U$
est dominant et que sa fibre g\'en\'erique est de dimension au moins 1. Alors, avec les notations ci-dessus :

 (i) Sous l'une des hypoth\`eses (1), (2), (3) sur $W$, on a (d).

(ii)  Sous l'une des hypoth\`eses (1), (2), (3), (4) sur $W$, on a (e).

 (iii) Sous l'hypoth\`ese (6) sur $X$, par exemple si $X$ est $k$-rationnelle, on a (a), (b) et (d).
 
(iv) Sous l'une des hypoth\`eses (6), (7) sur $X$, on a (b) et (d).

(v) Sous l'une des hypoth\`eses (6), (7), (8) sur $X$, on a (c) et (e).

(vi) Sous l'hypoth\`ese (5), on a (g).

(vii) Sous l'hypoth\`ese (5), si la fibre g\'en\'erique de $W \to U$  est g\'eom\'etriquement int\`egre,  on a (f).

(viii) Sous l'hypoth\`ese (9), on a (f) et (g).

  \end{theo}
 \begin{proof}

Pour \'etablir l'\'enonc\'e, on peut remplacer $W$ par un ouvert de Zariski et supposer que    $\psi:= \pi \circ g : W  \to U$
est lisse.
Pour les \'enonc\'es impliquant les hypoth\`eses (6), (7), (8) ou (9),
on prend $W=X$  et $g : X \to X$ l'identit\'e.

On consid\`ere le produit fibr\'e $Y=X\times_{U}W$ de $X$ et $W$
au-dessus de $U$.  
Via la projection $q: Y \to W$ sur le second facteur, c'est un $W$-sch\'ema ab\'elien
obtenu par changement de base \`a partir de $X \to U$.
 L'image du point g\'en\'erique de $W$ dans la fibre g\'en\'erique $X_{\eta}$ de
 $\pi$ par hypoth\`ese n'est pas un point ferm\'e de cette fibre g\'en\'erique.
 D'apr\`es le lemme \ref{generique}, ceci implique que la fibre g\'en\'erique de $q: Y \to W$ a 
 son rang de Mordell-Weil au moins \'egal \`a $r+1$. 
Pour $t\in W(k)$, la $k$-vari\'et\'e ab\'elienne $Y_{t}$ est isomorphe \`a la vari\'et\'e ab\'elienne $X_{\psi(t)}$.
Ces deux $k$-vari\'et\'es ab\'eliennes ont en particulier m\^{e}me rang de Mordell-Weil sur $k$.
Notons ${\mathcal S} \subset W(k)$ l'ensemble des $k$-points   $t \in W(k)$
dont la fibre  $Y_{t}$ a un rang de Mordell-Weil au moins \'egal \`a $r+1$.
Pour $t\in W(k)$, on a $t \in {\mathcal S} $ si et seulement si $\psi(t) \in {\mathcal R} \subset U(k)$.

D'apr\`es le th\'eor\`eme  \ref{neronserre}, le compl\'ementaire de $\mathcal{S}$ dans $W(k)$ est mince.
Sous l'hypoth\`ese (3) sur $W$,  le th\'eor\`eme \ref{mincefin} appliqu\'e \`a $W$
implique que l'adh\'erence de $\mathcal{S}$ dans $\prod_{v\in \Omega} W(k_{v})$ est ouverte.
Sous l'hypoth\`ese (4) sur $W$, le th\'eor\`eme  \ref{mincefin} appliqu\'e \`a $W$
assure au moins l'existence d'un ensemble fini $S_{0}$ de places tel
que  l'adh\'erence de $\mathcal{S}$ dans $\prod_{v \notin S_{0}} W(k_{v})$
est ouverte. Comme on a suppos\'e $W \to X$ lisse, pour tout ensemble fini
$T$ de places de $k$, 
les applications $$\prod_{v \in T} W(k_{v}) \to  \prod_{v \in T} U(k_{v})$$
sont ouvertes. Ainsi (3) implique (d),  et (4) implique (e).

Consid\'erons le cas particulier o\`u  le morphisme $g: W \to X$ est l'identit\'e de $X$
et $\psi=\pi : X \to U$. Dans ce cas, on dispose de la section nulle $\sigma : U \to X$
de $\psi= \pi: X \to U$.
On a $t \in {\mathcal S} \subset X(k) $ si et seulement si $\pi(t) \in {\mathcal R} \subset U(k)$.
Sous l'hypoth\`ese (6), resp. (7), resp. (8) sur $X$,
le th\'eor\`eme \ref{mincefin} (i), resp. (ii), resp. (iii)
 donne les r\'esultats de densit\'e sur $X$, qui induisent
 les r\'esultats voulus via les surjections  donn\'ees par $\pi$.

Montrons (vi), (vii) et (viii).
Sous l'hypoth\`ese (5) sur $W$ (ou (9), avec $W=X$),   l'ensemble $\mathcal{S} \subset W(k)$ n'est pas mince dans $W$, donc 
est dense dans $W$ pour la topologie 
de Zariski,  donc  $\mathcal{R} \subset U(k)$ qui contient son image par le morphisme dominant $\psi$
est  dense dans $U$  pour la topologie 
de Zariski.   
Si la fibre g\'en\'erique de $W \to U$ est g\'eom\'etriquement int\`egre,
le lemme \ref{mincegeomint} implique que l'image de  $\mathcal{S}$ dans $U(k)$ n'est pas mince dans $U$.
L'ensemble $\mathcal{R} \subset U(k)$ qui contient cette image n'est donc pas mince dans $U$.

L'\'enonc\'e (viii) est un cas particulier de (vi) et (vii).
  \end{proof} 
  
\begin{rema}  Supposons que $X_{\eta}$ est une vari\'et\'e ab\'elienne simple.
Soit $\rho \geq 1$ le rang du groupe ab\'elien de type fini ${\rm Hom}_{k(U)-{\rm gp}}(X_{\eta},X_{\eta})$.
On a alors $r=\rho.s$.
Sous les hypoth\`eses ci-dessus, on conclut \`a la densit\'e (en les divers sens)
des $k$-points dont la fibre est de rang au moins $\rho.(s+1)$.
Sous l'hypoth\`ese d'approximation faible faible pour $X$, ceci s'applique par exemple au cas d'une famille de courbes elliptiques  $X \to {\bf P}^1_{k}$
dont la fibre g\'en\'erique a multiplication complexe d\'efinie sur le corps $k({\bf P}^1)$.
Dans ce cas, $\rho=2$.
\end{rema}

 L'\'enonc\'e suivant est motiv\'e par des questions discut\'ees dans \cite{Maz} et \cite{R}
 (voir la proposition \ref{rohr} et la remarque \ref{rohrbis} ci-dessous).

\begin{theo}\label{pourrohrlich}
Soient $k$ un corps de nombres, $U$ une $k$-vari\'et\'e  lisse g\'eom\'etriquement int\`egre poss\'edant un $k$-point,
 et $\pi : X \to U$ un sch\'ema ab\'elien.
 Soit $r$ le rang   de Mordell-Weil de la fibre g\'en\'erique $X_{\eta}$ sur le corps $K=k(U)$.
 Soit ${\mathcal R}  \subset U(k)$ l'ensemble des points  $m\in U(k)$ 
 tels que le rang $r_{m}$ de la fibre $X_{m}(k)$ soit au moins \'egal \`a $r+1$.
   Soit $U \subset U_{c}$ une $k$-compactification lisse de $U$
 et
  $\pi_{c} : X_{c} \to U_{c}$ un $k$-morphisme de $k$-vari\'et\'es projectives lisses
  connexes \'etendant le morphisme $\pi = X \to U$ et \'equip\'e d'une section
  $\sigma_{c} : U_{c} \to X_{c}$ \'etendant la section nulle du sch\'ema ab\'elien $X \to U$.
 Supposons que la $k$-vari\'et\'e $X_{c}$ satisfait (BMF). Soit $S_{\infty}$
 l'ensemble des places archim\'ediennes de $k$.
L'ensemble ${\mathcal R}$ a une adh\'erence ouverte dans $\prod_{v \in \Omega} U_{c}(k_{v})$,
et son adh\'erence dans $\prod_{v \in S_{\infty}} U_{c}(k_{v})$
est une union finie de composantes connexes de ce produit.
\end{theo}
\begin{proof}
On garde les notations de la d\'emonstration pr\'ec\'edente, dans 
 le cas particulier o\`u  le morphisme $g: W \to X$ est l'identit\'e de $X$. On a les
 morphismes
$\psi=\pi : X \to U$ et $q: Y \to X$.
Rappelons que pour $t\in X(k)$, on 
  a  $t \in {\mathcal S} \subset X(k) $ si et seulement si $q(t) \in {\mathcal R} \subset U(k)$.
D'apr\`es le th\'eor\`eme \ref{mincefin} (iv),
 l'adh\'erence de ${\mathcal S} \subset X(k)$ dans $X_{c}(\A_{k})$
 co\"{\i}ncide avec l'ouvert $X_{c}(\A_{k})^{{\rm Br}(X_{c})}$.

   Comme l'\'evaluation d'un \'el\'ement du groupe de Brauer de $X_{c}$ 
   est localement constante sur chaque $X_{c}(k_{v})$ et donc constante sur chaque composante connexe de $X_{c}(k_{v})$ pour $v$ place archim\'edienne, ceci implique 
   que l'adh\'erence de ${\mathcal S}$ dans le produit
   $\prod_{v \in S_{\infty}} X_{c}(k_{v})$ est une union finie de composantes connexes.
   
   L'image par $\sigma_{c} = U_{c} \to X_{c}$ d'une composante connexe $V$
   de 
   $\prod_{v \in S_{\infty}} U_{c}(k_{v})$ est  un ensemble connexe et contient
   $\sigma_{c}(U_{c}(k))$, et en particulier au moins l'image d'un $k$-point de ${\mathcal R} \subset U(k)$.
   Cette image est donc contenue dans une composante connexe $Z$
   de  $\prod_{v \in S_{\infty}} X_{c}(k_{v})$ qui contient un point de ${\mathcal S} \subset X(k)$.
   Tout point de cette composante connexe $Z$, et en particulier tout point de $\sigma_{c}
  (\prod_{v \in S_{\infty}} U_{c}(k_{v}))$
    est donc dans l'adh\'erence de ${\mathcal S}$.
   L'adh\'erence de ${\mathcal R} =\pi_{c}({\mathcal S})$ est donc dense 
   dans $V$.
   \end{proof}

\begin{remas}\label{bilsalhinlou}
Soient $U$ un ouvert de la droite projective sur un corps de nombres $k$
 et   $\pi: X \to U$  une famille de courbes elliptiques, avec $X(k) \neq \emptyset$.

(i)  Pour $k=\Q$ le corps des rationnels, supposant la surface  $X$   $\Q$-rationnelle et la famille  non isotriviale,  H. Billard ([B], Th\'eor\`eme 4.1)
par un argument   impliquant  une comparaison de hauteurs globales sur $X$ avec les hauteurs
de Tate sur les fibres de $X \to U$
avait  \'etabli que l'ensemble des $\Q$-points $m$ avec $r_{m} \geq r+1$
 est infini.

(ii)  C. Salgado \cite[Cor. 1.2]{S} a \'etendu ce r\'esultat sur tout corps 
de nombres $k$ sous la simple 
 hypoth\`ese que la surface $X$ est $k$-unirationnelle. C'est une cons\'equence
 de \cite[Thm. 1.1]{S}, selon lequel, sous cette hypoth\`ese,  on peut 
 apr\`es un changement de base quasi-fini convenable $V \to U$,
 avec $V$ ouvert de $\P^1_{k}$, faire grandir d'au moins une unit\'e le rang de Mordell-Weil de la fibre g\'en\'erique.
  Le th\'eor\`eme  \ref{Hauptsatz} 
 ci-dessus g\'en\'eralise  \cite[Cor. 1.2]{S}. La d\'emonstration du th\'eor\`eme 
 \ref{Hauptsatz}  et un argument de sp\'ecialisation \`a la N\'eron doivent permettre 
d'\'etendre \cite[Thm. 1.1]{S} au cas d'une famille de vari\'et\'es ab\'eliennes
au-dessus d'un ouvert de $\P^1_{k}$.
  
(iii)  Lorsque la surface $X$ est $k$-rationnelle, sous certaines hypoth\`eses suppl\'ementaires sur $X$, C. Salgado  \cite{S} montre que l'ensemble des $k$-points $m \in U(k)$
 avec $r_{m} \geq r+2$ est infini. Un \'enonc\'e plus g\'en\'eral est obtenu par Loughran et Salgado \cite{LS} : il porte sur les surfaces elliptiques g\'eom\'etriquement rationnelles.

 (iv) Dans \cite{HS}, M. Hindry et C. Salgado \'etendent un certain nombre des  r\'esultats de \cite{S} au cas des
 familles de vari\'et\'es ab\'eliennes sur la droite projective. Leur th\'eor\`eme  1.4 est maintenant un cas particulier
 du th\'eor\`eme \ref{Hauptsatz} ci-dessus. Ils \'etablissent aussi la g\'en\'eralisation   de 
  \cite[Thm. 1.1]{S} sugg\'er\'ee dans (ii) ci-dessus. Dans la version initiale du pr\'esent article, je m'\'etais limit\'e \`a supposer
  $W \to X$ dominant. La g\'en\'eralisation au cas o\`u  $\pi : g(W)_{ad} \to U$
 est dominant  de dimension relative au moins \'egale \`a 1 a \'et\'e sugg\'er\'ee par une remarque dans \cite{HS}. 
  \end{remas}

\section{Applications}

A la suite d'un article de B. Mazur \cite{Maz},
 la variation du rang dans les tordues quadratiques d'une courbe elliptique donn\'ee
 a fait l'objet de plusieurs travaux.

Le r\'esultat suivant est sans doute bien connu.

\begin{prop}\label{degre1}
 Soient  $k$ un corps de nombres et $E$ une courbe elliptique d'\'equation $y^2=p(x)$,
avec $p(x)$ polyn\^{o}me s\'eparable de degr\'e 3.
L'ensemble des 
  $t_{0} \in k$ tels que le rang de Mordell-Weil de la courbe 
d'\'equation $t_{0}.y^2=p(x)$
soit au moins \'egal \`a 1  est dense dans  le produit topologique $\prod_{v \in \Omega} k_{v}$.
\end{prop}

\begin{proof} La surface d'\'equation $ty^2=p(x)$ est clairement $k$-birationnelle
\`a un plan affine et satisfait en particulier l'approximation faible.  Le
  th\'eor\`eme \ref{Hauptsatz} (iii) donne le r\'esultat.
\end{proof}

\begin{prop}\label{rohr}\label{degre2}
 Soit $E$ une courbe elliptique d'\'equation affine $y^2=p(x)$
 sur un corps de nombres $k$, avec $p(x)$ un polyn\^{o}me s\'eparable de degr\'e 3. Soit $d(t)$ un polyn\^{o}me quadratique
 s\'eparable. 
L'adh\'erence des $t_{0} \in k$  avec $ d(t_{0})\neq 0$ tels que 
 le rang de Mordell-Weil de la courbe  
 d'\'equation $d(t_{0}).y^2=p(x)$ soit au moins \'egal \`a 1
contient un ouvert de $\prod_{v\in \Omega} k_{v}$.
L'ensemble de ces $t_{0}$ est dense dans le produit
 $\prod_{v\in  S_{\infty}} k_{v}$. 
\end{prop}
\begin{proof} On peut supposer le polyn\^{o}me $d(t)$ donn\'e par 
$$d(t)= c(t^2-a)$$ avec $c.a \neq 0$. 
 La surface $X$ d'\'equation
 $ c(t^2-a)y^2=p(x)$ qui via la projection sur $t$ d\'efinit une famille
 de courbes elliptiques,
  est birationnelle \`a une surface d'\'equation affine
 $ u^2-av^2= cp(x)$. Comme $p(x)$ est un polyn\^{o}me de degr\'e 3,
 on reconna\^{\i}t l\`a l'\'equation affine d'une surface de Ch\^{a}telet, 
 dont de plus tout
 mod\`ele projectif et lisse poss\`ede un point $k$-rationnel.
 D'apr\`es \cite[Thm. 8.7]{CTSaSD}, une telle surface satisfait (BMF). 
L'\'enonc\'e suit alors du th\'eor\`eme \ref{pourrohrlich} et de la connexit\'e
de $\prod_{v\in  S_{\infty}}  \P^1(k_{v})$.
\end{proof}

\begin{rema}\label{rohrbis}
Cet \'enonc\'e g\'en\'eralise le th\'eor\`eme 
de Rohrlich  \cite[Theorem 3]{R} assurant que pour $k={\Q}$,
sous les hypoth\`eses de la proposition \ref{rohr}, s'il existe au moins un  $t_{0} \in {\Q}$ \`a fibre lisse de rang au moins 1, 
alors  l'ensemble de ces points
est dense dans ${\R}$. Pour \'etablir ce r\'esultat,   Rohrlich utilise deux fibrations elliptiques
diff\'erentes sur la surface $X$. 

Le lien entre les tordues quadratiques par un polyn\^{o}me de degr\'e 2 et les surfaces
de Ch\^{a}telet a \'et\'e observ\'e r\'ecemment par D. Loughran et C. Salgado \cite{LS}.
\end{rema}

\section{Appendice: Vari\'et\'es ab\'eliennes de grand rang}\label{grandrang}

Dans toute cette section, $k$ d\'esigne un corps de nombres.

Le th\'eor\`eme suivant est une extension formelle
d'un th\'eor\`eme de N\'eron \cite[Chap. IV]{N}.
 
\begin{theo}\label{neronetkummer}
Soit $A$ une vari\'et\'e ab\'elienne sur $k$.
Soit $r\geq 1$ un entier. Soit $X$ une $k$-vari\'et\'e projective,
lisse, g\'eom\'etriquement int\`egre. Supposons qu'il existe des ouverts non vides
$U \subset A^r$, $V \subset X$ et un morphisme fini \'etale $q= U \to V$ de degr\'e $d \geq 1$.
Soit ${\mathcal R} \subset V(k)$  l'ensemble des $k$-points $m$ dont la fibre $n:= q^{-1}(m)$
est un point ferm\'e $n$ de degr\'e $d$ sur $k$ tel que le groupe de Mordell-Weil
 $A(k(n))$ soit de rang au moins $r$.
 
Si la $k$-vari\'et\'e $X$ satisfait l'approximation faible faible, 
alors il existe un ensemble fini $S_{0}$ de places de $k$
tel que ${\mathcal R}$ est dense dans $\prod_{v \notin S_{0}} X(k_{v})$
et en particulier est dense  dans $X$ pour la topologie de Zariski.
\end{theo}
\begin{proof} Soit $L=k(A^r)$ et $K=k(X)$.
On consid\`ere la fibration
$A^{r+1} \to A^{r}$ donn\'ee par la projection sur les $r$ derniers facteurs. Le rang g\'en\'erique,
c'est-\`a-dire le rang du groupe $A(L)$ est au moins \'egal \`a $r$.
On consid\`ere la $K$-vari\'et\'e ab\'elienne $B:=R_{L/K}(A)$. Par une variante du lemme \ref{generique},
 le rang de
$B(K)=A(L)$ est au moins \'egal \`a $r$. On consid\`ere le $V$-sch\'ema ab\'elien
$R_{U/V}(A_{U})$.  Le th\'eor\`eme \ref{neronserre} 
donne que l'ensemble des $k$-points de $V$ dont la fibre est de rang au moins
$r$ est le compl\'ementaire d'un ensemble mince de  $V(k)$.
L'hypoth\`ese que $X$ satisfait l'approximation faible faible et
le th\'eor\`eme \ref{mincefin} donnent alors le r\'esultat.
\end{proof}

\begin{cor}\label{nerondegred} (N\'eron)
Soit $A$ une vari\'et\'e ab\'elienne sur $k$. 
Soit $r\geq 1$ un entier. Soit $d \geq  1$ le degr\'e d'une application rationnelle
g\'en\'eriquement finie de $A$ vers un espace projectif $\P^n_{k}$.
Il  existe une extension de corps $L/k$ compos\'ee
 de $r$ extensions de degr\'e $d$ telle que le rang de Mordell-Weil
de $A(L)$ soit au moins $r$.
\end{cor}
\begin{proof}
Il existe un ouvert    $V$ de $A$, un ouvert non vide $U \subset \P^n_{k}$,
et un morphisme fini \'etale $V \to U$ de degr\'e $d$.
On consid\`ere le morphisme fini \'etale $V^r \to U^r$.
L'image r\'eciproque d'un $k$-point de  $U^r$
est le spectre d'une $k$-alg\`ebre produit tensoriel de $r$ extensions s\'eparables
de $k$ chacune de degr\'e $d$. 
L'\'enonc\'e r\'esulte donc du th\'eor\`eme \ref{neronetkummer}.
 \end{proof}

Comme le note N\'eron, le cas particulier  suivant avait d\'ej\`a \'et\'e 
\'etabli par Billing \cite{Bg}.

\begin{cor}\label{billing}
Soit $E$ une courbe elliptique sur $k$.  

(i) Soit $r\geq 1$ un entier. 
Il existe une extension multiquadratique $L/k$ de degr\'e   $2^r$
 telle que le rang de Mordell-Weil de $E(L)$ soit au moins $r$.
 
 (ii) Soit $k_{2}$ le compos\'e de toutes
les extensions quadratiques de $k$.  Le rang de $E(k_{2})$ est infini.
\end{cor}

La m\^{e}me m\'ethode redonne le th\'eor\`eme de Holmes et Pannekoek \cite{HP}.
En se ramenant au  th\'eor\`eme de
Skorobogatov et Zarhin \cite{SZ} sur les surfaces $K3$,
Holmes et Pannekoek \cite[Prop. 13]{HP}  montrent que pour
toute vari\'et\'e de Kummer $X$ sur un corps de nombres $k$,
le quotient ${\rm Br}(X)/{\rm Br}(k)$ est fini.
Ceci implique    \cite[Prop. 14]{HP} que si l'obstruction de Brauer-Manin
est la seule pour les points rationnels sur $X$ alors 
l'approximation faible faible vaut pour $X$.

\begin{theo} \cite[Thm. 17]{HP}.
Soient $k$ un corps de nombres et $A$ une vari\'et\'e ab\'elienne.
Soit $r>0$ un entier. Soit $X_{r}$ un mod\`ele projectif et lisse
de la vari\'et\'e de Kummer quotient de $A^r$ par l'involution
antipodale. Si $X_{r}$ satisfait l'approximation faible faible,
alors il existe une tordue quadratique de $A$
sur $k$ qui est de rang au moins $r$.
En particulier, si la propri\'et\'e d'approximation faible faible vaut pour
toutes les vari\'et\'es de Kummer $X_{r}$, le rang des tordues quadratiques de $A$ n'est pas born\'e.
\end{theo}

\begin{proof}
On consid\`ere la famille de vari\'et\'es ab\'eliennes
$\pi: A^{r+1} \to  A^{r}$ donn\'ee par la projection sur les $n$ derniers facteurs.
Soit $L=k(A^r)$ et $K=k(X)$.
Soit $q: U \to V$ un morphisme fini \'etale de degr\'e 2,
avec $U \subset  A^r$ et $V \subset X$ des ouverts de Zariski convenables.
Soit $\rho$ le rang de $A(k)$.
 Le rang du groupe de Mordell-Weil $A(L)$ de
 la fibre g\'en\'erique $A_{L}$ est au moins \'egal \`a $r+\rho$. 
 Notons $B$ le $V$-sch\'ema ab\'elien $R_{U/V}(A_{U})$.

 On a $B(K)=A(L)$. Le rang de $B(K)$ est donc au moins $r+\rho$.
Le  th\'eor\`eme \ref{neronetkummer}  
montre que les $k$-points $m$  de 
$V$, de fibre $n=q^{-1}(m)$ un point ferm\'e de degr\'e 2,
 tels que le rang de la vari\'et\'e $B_{m}$ sur $k$
 et donc de $A(k(n))$  soit au moins \'egal \`a $r+\rho$ sont denses dans $V$
 pour la topologie de Zariski.
 Le rang de $A(k(n))$ est la somme de $\rho$  
 et du rang sur $k$ d'une tordue quadratique de $A$,
 dont le rang sur $k$ est donc au moins $r$.
\end{proof}

 \begin{rema}
 La preuve donn\'ee ici est l\'eg\`erement diff\'erente de celle de \cite{HP}.
 En particulier, il n'y a pas de discussion du comportement de la torsion
 des vari\'et\'es ab\'eliennes dans la famille, ni d'apparente utilisation de hauteurs.
  \end{rema}

\bigskip
 
{\bf Remerciements.} {\it Cette note a \'et\'e  a \'et\'e suscit\'ee par un expos\'e de C.~Salgado.
Je remercie D. Loughran pour diverses discussions.  L'appendice a \'et\'e suscit\'e par
 la remarque de B. Poonen que la 
m\'ethode du point g\'en\'erique m\`ene au corollaire \ref{nerondegred}.
Je remercie le rapporteur ou la rapporteuse de cet article pour sa lecture attentive.
Le travail a \'et\'e r\'ealis\'e lors du trimestre 
``A la red\'ecouverte des points rationnels'' organis\'e par D. Harari, E. Peyre et A. Skorobogatov \`a  l'Institut Henri Poincar\'e \`a Paris, d'avril \`a juillet 2019.

\end{document}